\documentclass[a4 paper,11pt]{amsart}

\usepackage{graphicx}
\usepackage{amsmath} %
\usepackage{amssymb} %
\usepackage{amscd} %
\usepackage{amsthm} %
\usepackage{mathrsfs} %
\usepackage{euscript}
\usepackage{eufrak}
\usepackage{ulem}
\usepackage[alphabetic]{amsrefs}
\usepackage{ascmac} 
\usepackage{comment}
\usepackage{bm}
\usepackage{picture}
\usepackage{caption}
\usepackage{tikz}
\usetikzlibrary{positioning, intersections,calc,arrows.meta, math}

\usepackage{hyperref}
\usepackage{xcolor}
\definecolor{unbleu}{rgb}{0.03, 0.15, 0.4}

\hypersetup{
pdfborder = {0 0 0},
colorlinks,
linkcolor=unbleu,
citecolor=unbleu,
urlcolor=unbleu
}

\usepackage{fancyhdr}
 \newtheorem{theorem}{Theorem}[section]
 \newtheorem{lemma}[theorem]{Lemma}

\newtheorem{maintheorem}{Theorem} 
  
\theoremstyle{definition}
\newtheorem{definition}[theorem]{Definition}

\newtheorem{example}[theorem]{Example}

\begin{document}

\title[]{On the stability of the penalty function for a nearest-neighbor $\mathbb{Z}^2$ subshift of finite type with the single-site fillability}

\author[C. Oguri]{Chihiro Oguri}
\address{Department of Mathematics, Ochanomizu University, 2-1-1 Otsuka, Bunkyo-ku, Tokyo, 112-8610, Japan}
\email{}

\author[M. Shinoda]{Mao Shinoda}
\address{Department of Mathematics, Ochanomizu University, 2-1-1 Otsuka, Bunkyo-ku, Tokyo, 112-8610, Japan}
\email{shinoda.mao@ocha.ac.jp}


\begin{abstract}

We investigate the stability of maximizing measures for a penalty function of a two-dimensional subshift of finite type, building on the work of Gonschorowski et al. \cite{GQS}. In the one-dimensional case, such measures remain stable under Lipschitz perturbations for any subshift of finite type. However, instability arises for a penalty function of the Robinson tiling, which is a two-dimensional subshift of finite type with no periodic points and zero entropy. This raises the question of whether stability persists in two-dimensional subshifts of finite type with positive topological entropy. In this paper, we address this question by studying a nearest-neighbor subshift of finite type satisfying the single-site fillability property.
Our main theorem establishes that, in contrast to previous results, a penalty function of such a subshift of finite type remains stable under Lipschitz perturbations.

\end{abstract}

\maketitle

\section{Introduction}

Ergodic optimization is the study of maximizing measures.
In its most basic form, let $T:X\rightarrow X$ be a continuous map on a compact metric space $X$ and for a continuous function $\varphi:X\rightarrow \mathbb{R}$ we consider the {\it maximum ergodic average}
\begin{align*}
    \beta(\varphi)=\sup_{\mu\in \mathcal{M}_T(X)}\int \varphi\ d\mu
\end{align*}
where $\mathcal{M}_T(X)$ is the space of $T$-invariant Borel probability measures on $X$ endowed with the weak*-topology.
An invariant measure which attains the maximum is called a {\it maximizing measure} for $\varphi$ and denote by $\mathcal{M}_{{\rm max}}(\varphi)$ the set of maximizing measures for $\varphi$.


The stability of maximizing measures for a penalty function of a subshift of finite type was established by Gonschorowski et al. \cite{GQS}.
A penalty function is defined on the forbidden set of a subshift of finite type, assigning a value of $0$ to admissible local configurations near the origin and $-1$ otherwise (see \S2 for more details). 
It is straightforward to see that every maximizing measure of a penalty function is supported on the given subshift of finite type. 
In the one-dimensional case, maximizing measures remain supported on the given subshift under Lipschitz perturbations for any subshift of finite type. 
However, in the two-dimensional case, there exists a subshift of finite type where this stability fails.

In \cite{GQS}, the authors highlight the difference between one and two dimensions, demonstrating that instability arises in the penalty function of the Robinson tiling, a two-dimensional subshift of finite type with no periodic points and zero entropy. In contrast, in the one-dimensional setting, stability results are established for subshifts of finite type that typically possess abundant periodic points and positive topological entropy. This contrast raises the natural question of whether stronger topological properties—such as positive topological entropy, rich periodic structure, or some mixing property—might guarantee stability in the two-dimensional case as well.
In this paper, we consider this question by investigating the penalty function on {\it nearest-neighbor $\mathbb{Z}^2$ subshifts of finite type} (n.n. SFTs) that satisfy the {\it single-site fillability} (SSF) property  (See for definitions in \S 2). This class includes, as a typical example, the {\it hard square shift}, a well-known two-dimensional subshift of finite type with positive entropy. Our main theorem establishes that, in contrast to the result on $\mathbb{Z}^2$ subshifts of finite type presented by Gonschorowski et al., the penalty function of any such system remains stable under Lipschitz perturbations.

Informally, a subshift of finite type is defined by specifying a finite set of finite ``forbidden patterns'' \(F\) made up of letters from an {\it alphabet} \(\mathcal{A}\), and defining \(X_F\) to be the set of configurations in \(\mathcal{A}^{\mathbb{Z}^2}\) in which no pattern from \(F\) appears (see §2 for more details). The set \(F\) is called a {\it forbidden set}.  
A subshift of finite type is called a {\it nearest-neighbor subshift of finite tyep} (n.n. SFT) if \(F\) can be chosen to consist only of patterns supported on pairs of adjacent sites.
For a n.n. SFT with a forbidden set $F$ we define the penalty function as follows:
 $$f(x)=\left\{
            \begin{array}{cc}
                -1&\mbox{if}\quad x_{(0,0)}x_{(0,1)}\in F\quad \mbox{or}\quad x_{(0,0)}x_{(1,0)}\in F,\\
                0&\mbox{otherwise}.
            \end{array}
            \right. .$$
Now we can state our main theorem.
\begin{maintheorem}
\label{stability}
    Let $X$ be a $\mathbb{Z}^2$ nearest-neighbor subshift of finite type with the single-site fillability and $f$ be the penalty function.
    Then there exists $\varepsilon>0$ such that for every Lipschitz continuous function $g$ with $\|f-g\|_{{\rm Lip}}<\varepsilon$, every maximizing measure of $g$ is supported on $X$.
\end{maintheorem}


We remark that the stability result for a n.n. SFT satisfying SSF is relatively straightforward since forbidden patterns can be easily eliminated by making local modifications guaranteed by SSF.
However, extending this result to more general SFTs appears to be substantially more difficult, as in such problems there is no method to extract the precise locations of bad words, and only their proportion can be accessed.
This lack of positional information prevents us from effectively handling configurations.
As a result, extending the stability result to systems with properties such as block gluing would likely require fundamentally new techniques.

For the remainder of this paper, we fix our notations and definitions in \S2 and provide the proof of the main theorem in \S3.

\section{Settings}
We denote the origin \((0,0)\) of \(\mathbb{Z}^2\) by \(\boldsymbol{0}\) to simplify notation.
For $\boldsymbol{u},\boldsymbol{v}\in \mathbb{Z}^d$ are said to be {\it adjacent} if $|\boldsymbol{u}-\boldsymbol{v}|=1$, where $|\boldsymbol{u}|=|u_1|+|u_2|$ for $\boldsymbol{u}=(u_1, u_2)\in \mathbb{Z}^2$.
The {\it boundary} of a set $S\subset \mathbb{Z}^2$, denoted by $\partial S$, is the set of $\boldsymbol{v}\in \mathbb{Z}^2\setminus S$ which are adjacent to some element of $S$.
For any $a,b\in \mathbb{Z}$ with $a<b$, we use $[a,b]$ to denote $\{a,a+1, \ldots, b\}$.
For each $n\geq 0$ define the box of size $n$ as
\begin{align*}
    \Lambda_n=[-n,n]\times [-n, n].
\end{align*}
The cardinality of $\Lambda_n$ is given by $\lambda_n=\#\Lambda_n=(2n+1)^2$.

Let $\mathcal{A}$ be a finite set, which we call an {\it alphabet}.
The $\mathbb{Z}^2$ {\it full shift} on $\mathcal{A}$ is the set $\mathcal{A}^{\mathbb{Z}^2}$, endowed with the product topology of the discrete topology.
Define a metric by
\begin{align*}
    d(\underline{x},\underline{y})=\left\{\begin{array}{cc}
      \frac{1}{2^i}   &  \underline{x}\neq \underline{y}\\
       0  & \mbox{otherwise}
    \end{array}
    \right.
\end{align*}
for $\underline{x},\underline{y}\in \mathcal{A}^{\mathbb{Z}^2}$ where $i=\inf \{\|\boldsymbol{u}\|_{\infty}: x_{\boldsymbol{u}}\neq y_{\boldsymbol{u}}\}$.
Then, this metric is compatible with the product topology.

For any full shift $\mathcal{A}^{\mathbb{Z}^2}$, we define the $\mathbb{Z}^2$-action $\{\sigma_{\boldsymbol{u}}\}_{\boldsymbol{u}\in \mathbb{Z}^2}$ on $\mathcal{A}^{\mathbb{Z}^2}$ as follows: for any $\boldsymbol{u}\in \mathbb{Z}^2$ and $\underline{x}\in \mathcal{A}^{\mathbb{Z}^2}$, $(\sigma^{\boldsymbol{u}}(\underline{x}))_{\boldsymbol{v}}=x_{\boldsymbol{u}+\boldsymbol{v}}$ for all $\boldsymbol{u}\in \mathbb{Z}^2$.

A {\it configuration} \(w\) on the alphabet \(\mathcal{A}\) is any mapping from a non-empty subset \(S\) of \(\mathbb{Z}^2\) to \(\mathcal{A}\), where \(S\) is called the {\it shape} of \(w\).  
If $S$ is finite, we call a configuration $w$ on $S$ is finite.
For any configuration \(w\) with shape \(S\) and any \(T \subset S\), we denote by \(w|_T\) the restriction of \(w\) to \(T\), i.e., the {\it subconfiguration} of \(w\) supported on \(T\).  
Let \(S, T \subset \mathbb{Z}^2\) with \(S \cap T = \emptyset\), and let \(w\) and \(w'\) be configurations with shapes \(S\) and \(T\), respectively.  
Then the {\it concatenation} of \(w\) and \(w'\) is the configuration on \(S \cup T\) defined by \((ww')|_S = w\) and \((ww')|_T = w'\), which is denoted by \(ww'\).
Let \(\mathcal{A}^*\) be the set of all configurations defined on finite subsets of \(\mathbb{Z}^2\).

A subset $X\subset \mathcal{A}^{\mathbb{Z}^2}$ is a {\it subshift} if it is closed and {\it shift-invariant}, i.e. for any $\underline{x}$ and $\boldsymbol{u}\in \mathbb{Z}^2$, $\sigma^{\boldsymbol{u}}(\underline{x})\in X$.
It is well known that any subshift can be also defined in terms of forbidden patterns: for a countable family $F$ of finite configurations, define
\begin{align*}
    X=X_F=\{\underline{x}\in \mathcal{A}^{\mathbb{Z}^2} \mid \sigma^{\boldsymbol{u}}(\underline{x})|_S\notin F\ \mbox{for all finite}\ S\subset \mathbb{Z}^d, \ \mbox{for all}\ \boldsymbol{u}\in \mathbb{Z}^2\}.
\end{align*}
Then $X=X_F$ is a subshift and all subshift can be represented in this way.


A subshift \(X\) is a {\it shift of finite type} (SFT) if there exists a finite collection \(F \subset \mathcal{A}^*\), called the {\it forbidden set}, such that \(X = X_F\).
If \(F\) consists only of configurations on pairs of adjacent sites, \(X\) is called a {\it nearest-neighbor shift of finite type} (n.n. SFT).  
Hereafter, for nearest-neighbor SFTs, we assume without further comment that their forbidden sets consist only of patterns defined on shapes of the form \(\{\boldsymbol{0}, \boldsymbol{0}+e_i\}\) for \(i=1,2\).

Let $X$ be a subshift.
A configuration $w$ on a $S\subset \mathbb{Z}^2$ is {\it globally admissible} for $X$ if there exists $\underline{x}\in X$ such that $x|_S=w$.
Let $F$ be a forbidden set defining $X$.
A configuration $w$ on $S\subset \mathbb{Z}^2$ is {\it locally admissible} for $X=X_F$ if for every $S'\subset S$, $w|_S\neq F$, up to translation.

Many combinatorial and topological mixing properties have been studied in \cite{MP15, Br16, BMP18}.
Here, we consider a strong combinatorial mixing property, the single-site fillability, introduced in \cite{MP15}.

\begin{definition}
    A n.n. SFT $X$ is {\it single-site fillable} (SSF) if for a finite forbidden set $F\subset \mathcal{A}^*$ such that $X=X_F$ and for every configuration $w$ on the shape on $ \mathcal{A}^{\partial\{\boldsymbol{u}\}}$ for some $\boldsymbol{u}\in \mathbb{Z}^2$, there exists $a\in \mathcal{A}^{\{\boldsymbol{u}\}}$ such that $w a$ is locally admissible.
\end{definition}

The property SSF is a generalization of the concept of a {\it safe symbol}.
The symbol $a\in \mathcal{A}^{\{\boldsymbol{0}\}}$ in the definition of SSF may depend on the configuration $w\in \mathcal{A}^{\partial\{\boldsymbol{0}\}}$.
If such a symbol \(a\) can be chosen independently of the surrounding configuration \(w\), then \(a\) is called a {\it safe symbol} (see for example \cite{MP15, Br16} for discussions on the relationships between mixing properties).
Note that for n.n. SFT $X=X_F$, a locally admissible configuration is globally admissible\cite{MP15}.

The following are typical examples of n.n. SFTs.

\begin{example}[Hard square shift]
    Let $\mathcal{A}=\{0,1\}$. Define $F\subset \mathcal{A}^*$ by
    \begin{align*}
        F=\bigcup_{i=1,2}\{w: \{\boldsymbol{0}, \boldsymbol{0}+e_i\}\rightarrow \mathcal{A}\mid w(\boldsymbol{0})=w(\boldsymbol{0}+e_i)=\{1\}\}.
    \end{align*}
    Then \(X_F\) is called the {\it hard square shift}, which consists of configurations \(\underline{x}\in \mathcal{A}^{\mathbb{Z}^2}\) with no adjacent 1’s.
\end{example}

\begin{example}[$k$-Checkerboard shift]
    Let $k\geq 2$, $\mathcal{A}=\{0,1,\ldots, k-1\}$. Define $F$ by
    \begin{align*}
        F=\bigcup_{i=1,2}\{w: \{\boldsymbol{0}, \boldsymbol{0}+e_i\}\rightarrow \mathcal{A} \mid w(\boldsymbol{0})=w(\boldsymbol{0}+e_i)\}.
    \end{align*}
    Then \(X_F\) is called the \(k\){\it -checkerboard shift}, which consists of configurations \(\underline{x}\in \mathcal{A}^{\mathbb{Z}^2}\) where no two adjacent sites have the same symbol.
\end{example}

The hard square shift has a safe symbol \(0\), as replacing any letter in a configuration by \(0\) yields an admissible configuration.  
In contrast, the \(k\)-checkerboard shift does not have a safe symbol for any \(k \geq 2\), since changing a letter arbitrarily may create adjacent sites with the same symbol.  
However, for \(k \geq 5\), the \(k\)-checkerboard shift satisfies SSF, since there are always enough remaining symbols to fill a site without violating adjacency constraints.


For a continuous function $f$ and a nonempty subset $T\subset \mathbb{Z}^2$ define a {\it Birkhoff sum} over $T$ by
\begin{align*}
    S_T f=\sum_{\boldsymbol{u}\in T} f\circ \sigma^{\boldsymbol{u}}.
\end{align*}

\section{Proof of the main theorem}

First we recall the following Lemma.

\begin{lemma}[A version of \textnormal{\cite[Lemma 2.1.]{GQS}}]
\label{lemma:Lip_bound}
    Let $J\subset X$ be a subset of a compact metric space $X$ and $f$ be a Lipschitz continuous function with $f|_J=c$ for some constant $c\in \mathbb{R}$.
    For $\varepsilon>0$ and a Lipschitz continuous function $g$ with $\|f-g\|_{{\rm Lip}}<\varepsilon$ we have
    \begin{align*}
        |g(x)-g(y)|<\varepsilon d(x,y)
    \end{align*}
    for all $x,y\in J$.
\end{lemma}
This lemma will be applied in our setting with $J=f^{-1}\{0\}$ and also with $J=f^{-1}\{- 1\}$.
With this preparation, we now proceed to the proof of our main theorem.
A key feature of this proof is its extension of the coupling and splicing argument, as well as the "path-wise surgery" technique from \cite{GQS}, to a two-dimensional case.

\begin{proof}[Proof of Theorem \ref{stability}]
Let $\varepsilon=\frac{1}{64}$ and $g$ be a Lipschitz function with $\|f-g\|_{{\rm Lip}}<\varepsilon$.

Set $I=f^{-1}\{0\}$.
Since the set of maximizing measures is convex and closed, it suffices to prove the result for ergodic measures.
Let $\mu$ be an ergodic invariant measure supported on $X^c$.

\noindent
(Case 1). $\mu(I^c)\geq 1/2$.

For every $\underline{x}\in \mathcal{A}^{\mathbb{Z}^2}$ we have $|f(x)-g(x)|<\varepsilon$ and
$\int f\ d\mu=-\mu(I^c)\leq -1/2$, then we have
\begin{align*}
    \int g\ d\mu=\int f\ d\mu+\int (g-f)\ d\mu\leq -\frac{1}{2}+\varepsilon=-\frac{33}{64}.
\end{align*}
On the other hand, for an invariant measure $\nu$ supported on $X$ we have $\int f\ d\nu=0$.
Hence we have
\begin{align*}
    \int g\ d\nu=\int f\ d\nu+\int (g-f)\ d\nu=\int (g-f)\ d\nu\geq -\varepsilon= -\frac{1}{64},
\end{align*}
which completes the proof.

\noindent
(Case 2). $\mu(I^c)\leq 1/2$.

Let $\underline{x}$ be a generic point for the measure $\mu$.
Let $S^0=\{\boldsymbol{0}\}$ if $\underline{x}\in I^c$, $S^0=\emptyset$ otherwise.
For each $i\in \mathbb{N}$ let
\begin{align*}
    S^i=\{\boldsymbol{u}\in \Lambda_i\setminus \Lambda_{i-1} \mid \sigma^{\boldsymbol{u}}\underline{x}\in I^c\}
    =S^{t(i)}\sqcup S^{b(i)} \sqcup S^{r(i)} \sqcup S^{l(i)}
\end{align*}
where
\begin{align*}
    S^{t(i)}&=\{(u_1, u_2)\in S^i \mid u_2=i\}, \quad
    S^{b(i)}=\{(u_1, u_2)\in S^i \mid u_2=-i\},\\
    S^{r(i)}&=\{(u_1, u_2)\in S^i \mid u_1=i, u_2\notin\{i,-i\}\},\\
    S^{l(i)}&=\{(u_1, u_2)\in S^i \mid u_1=-i, u_2\notin\{i,-i\}\}.
\end{align*}

\begin{figure}
     \centering
     \includegraphics[width=0.5\linewidth]{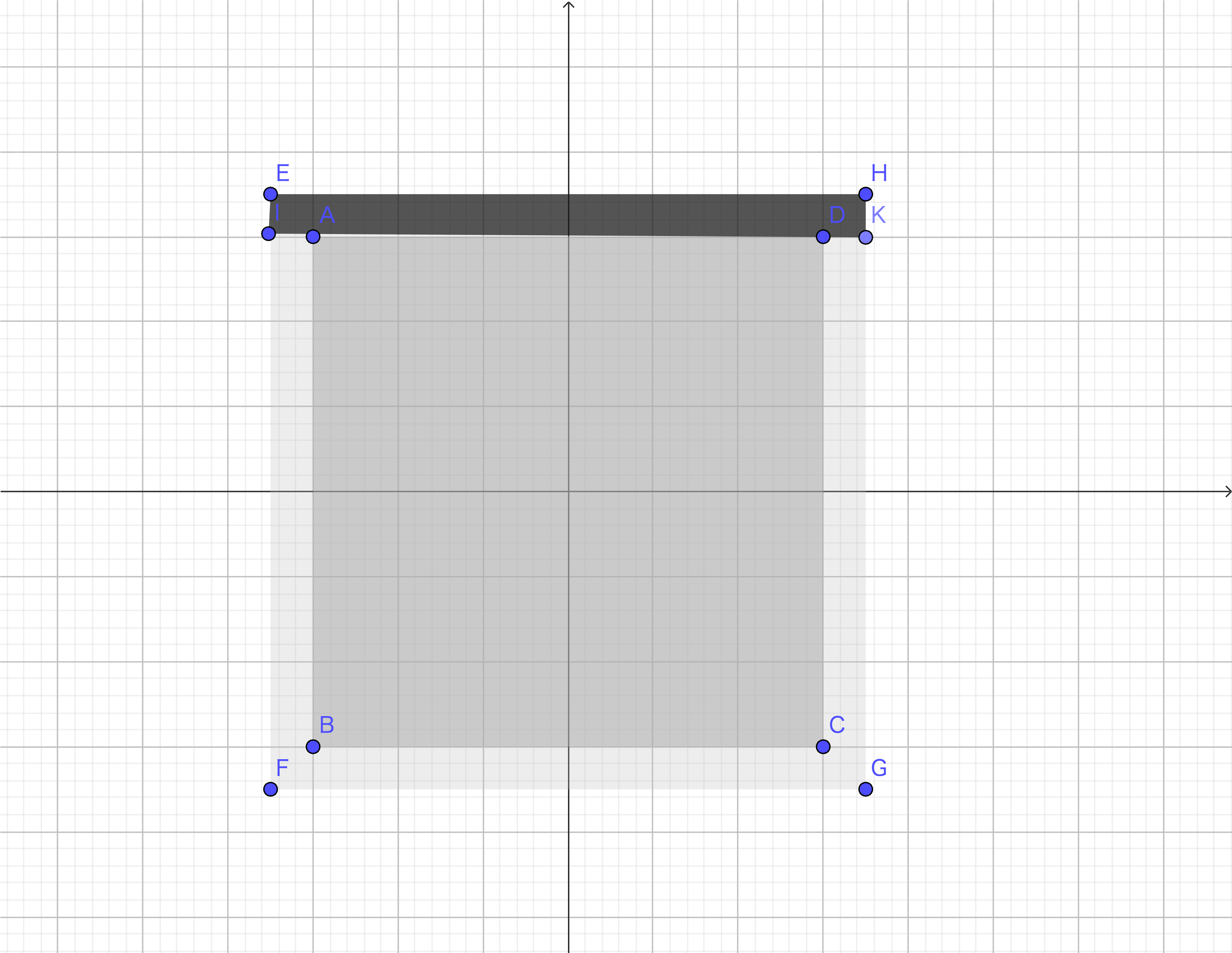}
     \caption{}
     \label{fig:enter-label}
 \end{figure}

For each \(\tau \in \{t(i), b(i), r(i), l(i)\}\) and
$\boldsymbol{u}, \boldsymbol{v}\in S^{\tau}$ set a relation $\boldsymbol{u} r\boldsymbol{v}$ by if they are adjacent, that is, 
\begin{align*}
    \boldsymbol{u}r\boldsymbol{v}\Leftrightarrow |\boldsymbol{u}-\boldsymbol{v}|=1.
\end{align*}
Moreover define the equivalent relation $\sim$ on $S^{\tau}$ by
\begin{align*}
    \boldsymbol{u}\sim \boldsymbol{v} \Leftrightarrow \ \mbox{there exist}\ \boldsymbol{w}^1, \ldots, \boldsymbol{w}^p\in S^{\tau}\ \mbox{such that}\ \boldsymbol{u}r\boldsymbol{w}^1; \boldsymbol{w}^1r \boldsymbol{w}^2; \cdots; \boldsymbol{w}^p r \boldsymbol{v}.
\end{align*}
Then we get the segments of bad words on $[-i,i]\times \{i\}$.
Set $S^{\tau}/\sim=\{B_k^{\tau}\}_{k=1}^{K_{\tau}}$ where the indices $k$ increase from left to right on the top and bottom sides, and from bottom to top on the right and left sides.
Setting
\[
\alpha_k^\tau = 
\begin{cases}
\min\{u_1 \mid (u_1,u_2)\in B_k^\tau\} & \text{if } \tau\in\{t(i),b(i)\},\\
\min\{u_2 \mid (u_1,u_2)\in B_k^\tau\} & \text{if } \tau\in\{r(i),l(i)\},
\end{cases}
\]
and
\[
\beta_k^\tau = 
\begin{cases}
\max\{u_1 \mid (u_1,u_2)\in B_k^\tau\} & \text{if } \tau\in\{t(i),b(i)\},\\
\max\{u_2 \mid (u_1,u_2)\in B_k^\tau\} & \text{if } \tau\in\{r(i),l(i)\},
\end{cases}
\]
we have
\[
B_k^\tau =
\begin{cases}
[\alpha_k^\tau,\beta_k^\tau]\times\{i\} & \text{if } \tau=t(i),\\
[\alpha_k^\tau,\beta_k^\tau]\times\{-i\} & \text{if } \tau=b(i),\\
\{i\}\times[\alpha_k^\tau,\beta_k^\tau] & \text{if } \tau=r(i),\\
\{-i\}\times[\alpha_k^\tau,\beta_k^\tau] & \text{if } \tau=l(i).
\end{cases}
\]

Before proving the main theorem, we show the following lemma, which will be used to apply the SSF property.

\begin{lemma}
\label{lem:replace}
    For \(i \in \mathbb{Z}\), an interval \([\alpha, \beta] \subset \mathbb{Z}\), and a configuration \(z\) on \(\partial ([\alpha,\beta]\times\{i\}) \cup ([\alpha,\beta]\times\{i\})\),  
    there exists a configuration \(w\) on \([\alpha,\beta]\times\{i\}\) such that \(z|_{\partial ([\alpha,\beta]\times\{i\})}w\) is locally admissible.

    The same holds for a configuration on $z$ on \(\partial(\{i\}\times[\alpha,\beta])\cup(\{i\}\times[\alpha,\beta])\).
\end{lemma}

\begin{proof}
    Using SSF inductively, we construct \(w\) as follows.  
    First, replace \(z|_{\{(\alpha, i)\}}\) by a symbol \(a\) such that \(z|_{\partial{\{(\alpha, i)\}}}a\) is admissible.  
    Let \(z^{(0)}\) denote the resulting configuration on \(\partial([\alpha, \beta]\times\{i\})\cup ([\alpha,\beta]\times\{i\}) \).  

    Then, for each \(j \in \{1, 2, \ldots, \beta - \alpha\}\), replace \(z^{(j-1)}|_{\{(\alpha+j, i)\}}\) by a symbol \(a\) such that \(z^{(j-1)}|_{\partial{\{(\alpha+j, i)\}}}a\) is admissible.  
    Define \(z^{(j)}\) as the configuration obtained after this replacement.  

    Finally, set
    \[
        w = z^{(0)}_{(\alpha, i)}z^{(1)}_{(\alpha+1, i)} \cdots z^{(\beta-\alpha)}_{(\beta, i)}.
    \]
    This \(w\) satisfies the desired property.
\end{proof}

Then we define a sequence of configurations on $\mathbb{Z}^2$ inductively as follows.  
First, set
\[
\underline{x}^{(-1)} := \underline{x}.
\]
Next define $\underline{x}^{(i)}$ inductively by
\begin{align*}
    x_{\boldsymbol{u}}^{(i)}=x_{\boldsymbol{u}}^{(i-1)}\quad\mbox{if}\quad \boldsymbol{u}\notin S^i
\end{align*}
and replace the configuration on $S^i$ by using Lemma \ref{lem:replace}.

Then, by definition, there is no bad word in the configuration \(\underline{x}^{(k)}|_{\Lambda_k}\) and the sequence \(\{\underline{x}^{(k)}\}_{k=0}^\infty\) converges.
Denote by the limit \(\tilde{x}\), and it is clear that \(\tilde{x}\in X\).


Fix sufficiently large $N\geq 1$.
We now consider the difference between the Birkhoff sums of $\underline{x}$ and $\tilde{\underline{x}}$ over $\Lambda_N$ :
\begin{align}
    S_{\Lambda_N}g(\underline{x})-S_{\Lambda_N}g(\tilde{\underline{x}})
    &=\sum_{i=0}^{N}\left(S_{\Lambda_N}g(\underline{x}^{i-1})-S_{\Lambda_{N}}g(\underline{x}^{i})\right)
    +S_{\Lambda_N}g(\underline{x}^{(N)})-S_{\Lambda_N}g(\tilde{\underline{x}}).
    \label{gap_with_periodic}
\end{align}
The last two terms can be bounded as follows:
\begin{align}
    S_{\Lambda_N}g(\underline{x}^{(N)})-S_{\Lambda_N}g(\tilde{\underline{x}})
    &=2\sum_{i=1}^{N}\frac{\#(\Lambda_{N+1}\setminus \Lambda_N)}{2^i}
    \leq 2(8N+1).
    \label{third}
\end{align}

In order to provide an upper bound for the summation term, we analyze the difference between the Birkoff sums of $\underline{x}^{i-1}$ and $\underline{x}^i$ over $\Lambda_N$ by considering contributions from bad words and others separately.

Fix $i\geq 0$.
First we divide $\Lambda_N$ into four regions $\Lambda_N=R^{t(i)}\sqcup R^{b(i)}\sqcup R^{r(i)}\sqcup R^{l(i)}$ such that
\begin{align*}
    R^{t(i)}&=[-N, N]\times [i, N]\cup \{(u_1,u_2)\mid u_1\in[-i, i], |u_1|\leq u_2\leq i-1\}\\
    R^{b(i)}&=[-N,N]\times[-N, -i]\cup \{(u_1, u_2) \mid u_1\in[-i,i], -|u_1|\geq u_2\geq -i+1, (u_1,u_2)\neq \boldsymbol{0}\}\\
    R^{r(i)}&=[i,N]\times [-i+1,i-1]\cup \{(u_1, u_2)\mid u_2\in [-i+1, i -1] ,|u_2|<u_1\leq i-1\}\\
    R^{l(i)}&=[-i, N]\times[-i+1, i-1]\cup \{(u_1, u_2)\mid u_2\in [-i+1, i-1], -|u_2|>u_1\geq -i+1\},
\end{align*}
(see also Figure \ref{fig:region}).

\begin{figure}
    \centering
    \includegraphics[width=0.5\linewidth]{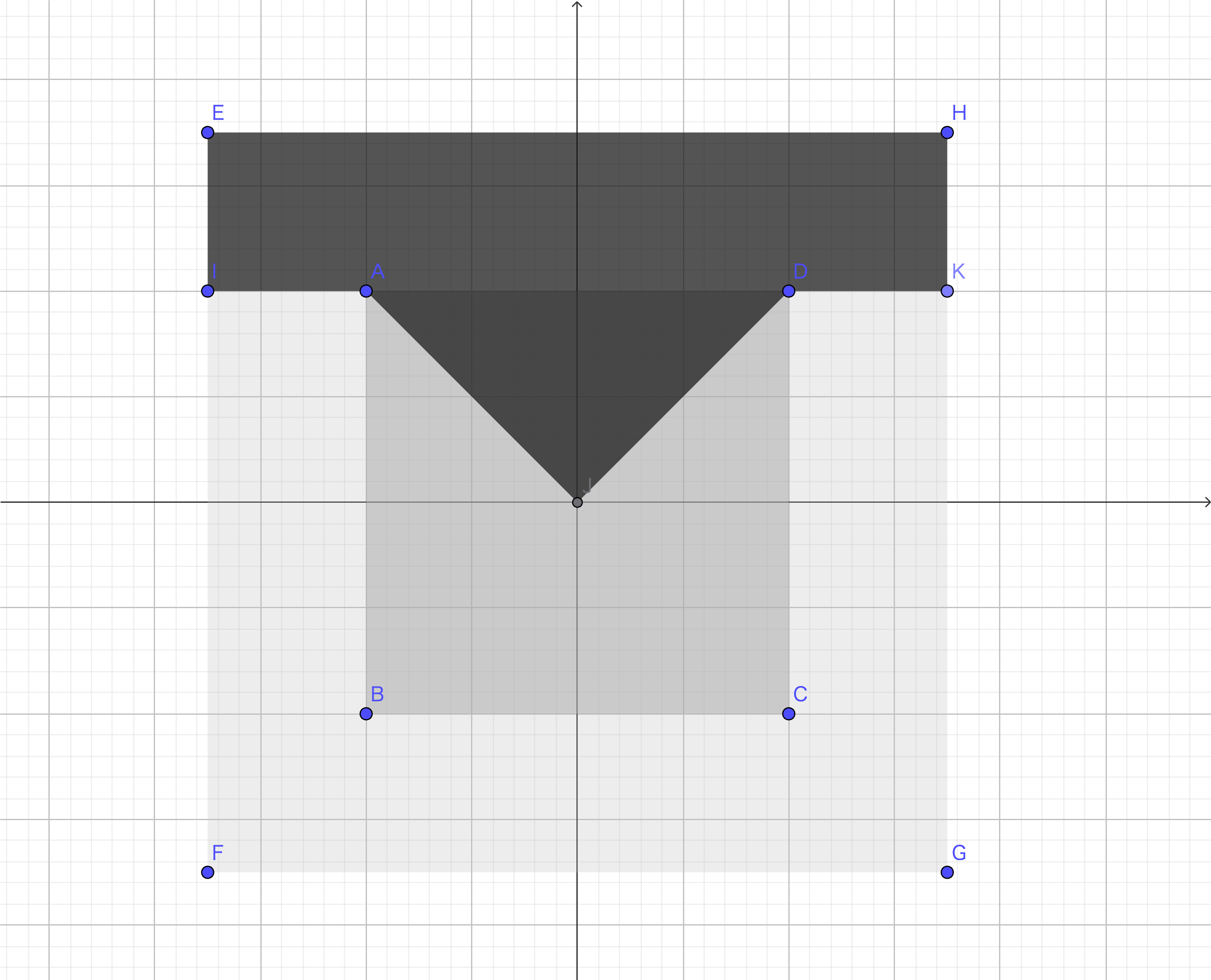}
    \caption{The square \(ABCD\) represents \(\Lambda_i\), the square \(EFGH\) represents \(\Lambda_N\), and the polygon \(EIAJBKH\) represents \(R^{t(i)}\).}
    \label{fig:region}
\end{figure}

\underline{Estimate on bad words:}
For each $\tau \in \{t(i), b(i), r(i),l(i)\}$ and $\boldsymbol{u}\in B^{\tau}_k$ for some $1\leq k\leq K_{\tau}$, we have 
\begin{align*}
    g(\sigma^{\boldsymbol{u}}\underline{x}^{(i-1)})<-1+\varepsilon,
    \quad \mbox{and}\quad
    g(\sigma^{\boldsymbol{u}}\underline{x}^{(i)})>-\varepsilon.
\end{align*}
Hence we have
\begin{align*}
    S_{B_k^{\tau}} g(\underline{x}^{i-1})-S_{B_k^{\tau}}g(\underline{x}^{i})
    <|B_k^{\tau}|(-1+\varepsilon)+|B_k^\tau|\varepsilon
    =|B_k^\tau|(-1+2\varepsilon)
\end{align*}
where $|E|$ denotes the cardinality of $E$.

\underline{Estimate on unchanged words:}

For evaluating the difference between the Birkhoff sums of \(\underline{x}^{(i-1)}\) and \(\underline{x}^{(i)}\) over “unchanged words,” we use an upper bound on the distance between \(\sigma^{\boldsymbol{u}}\underline{x}^{(i-1)}\) and \(\sigma^{\boldsymbol{u}}\underline{x}^{(i)}\) for each \(\boldsymbol{u} \in \Lambda_N \setminus S^i\).  
By Lemma~\ref{lemma:Lip_bound}, this upper bound depends on the distance to bad words when \(f(\sigma^{\boldsymbol{u}}\underline{x}^{(i-1)}) = f(\sigma^{\boldsymbol{u}}\underline{x}^{(i)})\).  

By the definition of the penalty function, we have \(f(\sigma^{\boldsymbol{u}}\underline{x}^{(i-1)}) \neq f(\sigma^{\boldsymbol{u}}\underline{x}^{(i)})\) if \(\boldsymbol{u} \in S^i\) or \(\boldsymbol{u} = \boldsymbol{v} - e_i\) for some \(\boldsymbol{v} \in S^i\) and \(i = 1, 2\).  

In the latter case, we have
\begin{align*}
g&(\sigma^{\boldsymbol{u}}\underline{x}^{(i-1)}) - g(\sigma^{\boldsymbol{u}}\underline{x}^{(i)}) \\
&\quad = g(\sigma^{\boldsymbol{u}}\underline{x}^{(i-1)}) - f(\sigma^{\boldsymbol{u}}\underline{x}^{(i-1)}) 
+ f(\sigma^{\boldsymbol{u}}\underline{x}^{(i-1)}) - f(\sigma^{\boldsymbol{u}}\underline{x}^{(i)})
 + f(\sigma^{\boldsymbol{u}}\underline{x}^{(i)}) - g(\sigma^{\boldsymbol{u}}\underline{x}^{(i)}).
\end{align*}
Since the change from \(\underline{x}^{(i-1)}\) to \(\underline{x}^{(i)}\) does not introduce any new bad words, we obtain
\[
g(\sigma^{\boldsymbol{u}}\underline{x}^{(i-1)}) - g(\sigma^{\boldsymbol{u}}\underline{x}^{(i)}) \leq \varepsilon\, d(\sigma^{\boldsymbol{u}}\underline{x}^{(i-1)}, \sigma^{\boldsymbol{u}}\underline{x}^{(i)}).
\]

First we consider $R^{t(i)}$ and let $\beta_0^{t(i)}=-N-1$ and $\alpha_{K_{t(i)}+1}^{t(i)}=\beta_{K_{t(i)}+1}^{t(i)}=N+1$.
Then for each $1\leq k\leq K_{t(i)}$, set $c^{t(i)}_k=\lfloor \frac{\alpha_k^{t(i)}-\beta_{k-1}^{t(i)}}{2}\rfloor(> 0)$
and the sets
\begin{align*}
    G_{k}^{t(i)}
    &=([\beta_{k-1}^{t(i)},\beta_k^{t(i)}]\times [i,N])\setminus ([\alpha_k^{t(i)}, \beta_k^{t(i)}]\times \{i\})\\
    &=([\beta_{k-1}^{t(i)},\beta_k^{t(i)}]\times [i,N])\setminus B_k^{t(i)}.
\end{align*}
Remark that we have
\begin{align*}
    \bigcup_{k=1}^{K_{t(i)}+1} G_k^{t(i)}
    &=[-N, N]\times[i,N]\setminus \left(\bigcup_{k=1}^{K_{t(i)}}B_k^{t(i)}\right)\\
    &=R^{t(i)}\setminus \left(\left(\bigcup_{k=1}^{K_{t(i)}}B_k^{t(i)}\right) \cup \{(u_1,u_2)\mid u_1\in[-i, i], |u_1|\leq u_2\leq i-1\}\right).
\end{align*}

Take $1\leq k\leq K_{t(k)}$ such that $c^{t(i)}_k<N-i$.
For $\boldsymbol{u} \in G_k^{t(i)}$, the distance  
$
d(\sigma^{\boldsymbol{u}}\underline{x}^{(i-1)}, \sigma^{\boldsymbol{u}}\underline{x}^{(i)})
$
is determined by three cases, and the computation is divided into four regions:
\begin{align*}
    A&=\{(u_1,u_2)\mid\beta^{t(i)}_{k-1}+1\leq u_1\leq\beta^{t(i)}_{k-1}+c_k^{t(i)}, -i\leq u_2\leq -i+u_1-(\beta^{t(i)}_{k-1}+1) \}\\
    &\quad  \cup \{(u_1,u_2)\mid\beta^{t(i)}_{k-1}+c_k^{t(i)}+1\leq u_1\leq \alpha_k^{t(i)}-1, -i\leq u_2\leq -i+c_k^{t(i)}-1-u_1+(\beta^{t(i)}_{k-1}+c_k^{t(i)}+1) \};\\
    B&=\{(u_1,u_2)\mid \beta^{t(i)}_{k-1}+1\leq u_1\leq\beta^{t(i)}_{k-1}+c_k^{t(i)},-i+u_1-(\beta^{t(i)}_{k-1}+1)\leq u_2\leq -i+c_k^{t(i)}-1\}\\
    &\quad \cup \{(u_1,u_2)\mid \beta^{t(i)}_{k-1}+c_k^{t(i)}+1\leq u_1\leq \alpha_k^{t(i)}-1,-i+u_1-(\alpha_k^{t(i)}-1)\leq u_2\leq -i+c_k^{t(i)}-1\};\\
    C&=\{(u_1, u_2)\mid\alpha_k^{t(i)}\leq u_1\leq \beta_k^{t(i)}, -i+1\leq u_2\leq c_k^{t(i)}\};\\
    D&=[\beta^{t(i)}_{k-1}+1, \beta_k^{t(i)}-1]\times [-i+c_k^{t(i)}+1, N].
\end{align*}
To illustrate this, we assign letters to each area as shown in Figure \ref{distance_bad_words}.  
The red graph represents a path where both $u_1$ and $u_2$ increase by 1 at each step.  

For $(u_1, u_2)$ in the region $A$, the horizontal distance from the bad words is the determining factor.
Specifically,
\begin{align*}
     d(\sigma^{(u_1, u_2)}\underline{x}^{(i-1)}, \sigma^{(u_1, u_2)}\underline{x}^{(i)})=\left\{\begin{array}{ccc}
        \frac{1}{2^{u_1 - \beta_{k-1}^{t(i)}}} &\mbox{if} & u_1  \leq \beta_{k-1}^{t(i)} +c_k^{t(i)}\\
         \frac{1}{2^{\tilde{\alpha}_{n}^{-i} - u_1}}  &  \mbox{if} & u_1  > \beta_{k-1}^{t(i)} +c_k^{t(i)}.
     \end{array}
    \right.
\end{align*}
For $(u_1,u_2)$ in the regions $B,C$ and $D$, the vertical distance is the determining factor.
Specifically,
   \[
   d(\sigma^{(u_1, u_2)}\underline{x}^{(i-1)}, \sigma^{(u_1, u_2)}\underline{x}^{(i)}) = \frac{1}{2^{u_2}}.
   \]

Taking into account the symmetry of regions $A$ and $B$, we compute as follows.
By Lemma \ref{lemma:Lip_bound}, we obtain the following bound:
\begin{align}
            S_{G_{k}^{t(i)}}g(\underline{x}^{(i-1)})-S_{G_{k}^{t(i)}}g(\underline{x}^{(i)})&< \left( \sum_{\ell=1}^{c_k^{t(i)}} 2\ell\cdot \frac{1}{2^\ell}+ \sum_{\ell=1}^{c_k^{t(i)}}2\ell\cdot\frac{1}{2^\ell}+\sum_{\ell=1}^{c_k^{t(i)}}(\beta^{t(i)}_{k}-\alpha^{t(i)}_{k})\frac{1}{2^\ell}\right.\notag\\
            &\hspace{3em}\left.+\sum_{\ell=c_k^{t(i)}}^{N-i}(\beta^{t(i)}_{k}-\beta^{t(i)}_{k-1})\frac{1}{2^\ell}+\sum_{\ell=1}^{c_k^{t(i)}}\frac{2}{2^\ell}\right)\varepsilon\notag\\
            &<\left(\sum_{\ell=1}^{c_k^{t(i)}}\frac{4\ell+2}{2^\ell}+(\beta^{t(i)}_k-\alpha_k^{t(i)})\sum_{\ell=1}^{N-i}\frac{1}{2^\ell}+\sum_{\ell=c_k^{t(i)}+1}^{N-i} \frac{(\alpha_k^{t(i)}-\beta^{t(i)}_{k-1})}{2^\ell}\right)\varepsilon\notag\\
            &\hspace{44pt} (\because \beta^{t(i)}_k-\beta^{t(i)}_{k-1}=\beta^{t(i)}_k-\alpha_k^{t(i)}+\alpha_k^{t(i)}-\beta^{t(i)}_{k-1})\notag\\
            &\leq\left(\sum_{\ell=1}^{c_k^{t(i)}}\frac{4\ell+2}{2^\ell}+(\beta^{t(i)}_k-\alpha_k^{t(i)})\sum_{\ell=1}^{N-i}\frac{1}{2^\ell}+\sum_{\ell=c_k^{t(i)}+1}^{N-i} \frac{2\ell}{2^\ell}\right)\varepsilon\notag\\
            &\hspace{44pt}(\because \alpha_k^{t(i)}-\beta^{t(i)}_{k-1} \leq 2c_k^{t(i)}+1)\notag\\
            &\leq \left(\sum_{\ell=1}^{N-i}\frac{4\ell+2}{2^\ell}+(\beta^{t(i)}_k-\alpha_k^{t(i)})\sum_{\ell=1}^{N-i}\frac{1}{2^\ell}\right)\varepsilon\notag\\
            &=\left(14+(\beta^{t(i)}_k-\alpha_k^{t(i)})\right)\varepsilon.
            \label{even}
\end{align}

        \begin{figure}[h]
            \tiny
            \centering
            \begin{tikzpicture}

                \draw[-,>=stealth,semithick] (-4.5,0)--(6,0) ; 
                \draw[-,>=stealth,semithick] (6.25,0)--(7,0) node[right]{$-i$}; 
                \draw[-,=stealth,semithick] (-2,-1)--(-2,0.9) ; 
                \draw[-,=stealth,semithick] (-2,1.1)--(-2,2.9) ; 
                \draw[-,=stealth,semithick] (-2,3.1)--(-2,4) ; 
                \draw[-,=stealth,dashed] (0,-1)--(0,2) ; 
                \draw[-,=stealth,semithick] (5,-1)--(5,0.9) ; 
                \draw[-,=stealth,semithick] (5,1.1)--(5,2.9) ; 
                \draw[-,=stealth,semithick] (5,3.1)--(5,4) ; 
                \draw[-,=stealth,dashed] (-2,2)--(5,2) ; 
                \draw[-,=stealth,dashed] (2,0)--(2,2) ; 
                \draw (2,0) node[above right]{$\tilde{\alpha}_{n}^{-i}$}; 
                \draw (5,0) node[above right]{$\tilde{\beta}_{n}^{-i}$}; 
                \draw (-2,0) node[above left]{$\tilde{\beta}_{n-1}^{-i}$}; 
                \draw (-2,2) node[left]{$c_{n}^{-i}$}; 
                \draw (0,0) node[below right]{$\tilde{\beta}_{n-1}^{-i}+c_n^{-i}$}; 
                \draw (6.12,0) node{\rotatebox{90}{$\approx$}}; 
                \draw (-2,1) node{$\approx$}; 
                \draw (-2,3) node{$\approx$}; 
                \draw (5,1) node{$\approx$}; 
                \draw (5,3) node{$\approx$}; 
                \draw (-2,3.5) node[left]{$N$}; 
                \draw (6.5,0) node[below]{$N$}; 
                \draw[red] (-2,0)--(-1.75,0) ;
                \draw[red] (-1.75,0)--(-1.75,0.25) ;
                \draw[red] (-1.75,0.25)--(-1.5,0.25) ;
                \draw[red] (-1.5,0.25)--(-1.5,0.5) ;
                \draw[red] (-1.5,0.5)--(-1.25,0.5) ;
                \draw[red] (-1.25,0.5)--(-1.25,0.75) ;
                \draw[red] (-0.75,1.25)--(-0.5,1.25) ;
                \draw[red] (-0.5,1.25)--(-0.5,1.5) ;
                \draw[red] (-0.5,1.5)--(-0.25,1.5) ;
                \draw[red] (-0.25,1.5)--(-0.25,1.75) ;
                \draw[red] (-0.25,1.75)--(0,1.75) ;
                \draw[red] (0,1.75)--(0,2) ;
                
                \draw[red] (2,0)--(1.75,0) ;
                \draw[red] (1.75,0)--(1.75,0.25) ;
                \draw[red] (1.75,0.25)--(1.5,0.25) ;
                \draw[red] (1.5,0.25)--(1.5,0.5) ;
                \draw[red] (1.5,0.5)--(1.25,0.5) ;
                \draw[red] (1.25,0.5)--(1.25,0.75) ;
                \draw[red] (1,1) node{$\ddots$}; 
                \draw[red] (0.75,1.25)--(0.5,1.25) ;
                \draw[red] (0.5,1.25)--(0.5,1.5) ;
                \draw[red] (0.5,1.5)--(0.25,1.5) ;
                \draw[red] (0.25,1.5)--(0.25,1.75) ;
                \draw[red] (0.25,1.75)--(0,1.75) ;
                \draw[red] (0,1.75)--(0,2) ;
                \draw (-0.75,0.5) node{A}; 
                \draw (0.75,0.5) node{A}; 
                \draw (-1.25,1.25) node{B}; 
                \draw (1.25,1.25) node{B}; 
                \draw (2,3) node{D}; 
                \draw (3.5,1) node{C}; 
            \end{tikzpicture}
            \caption{The value of point in good block}
            \label{distance_bad_words}
        \end{figure}

For $1\leq k\leq K^{t(i)}$ such that $N-i\leq c_k^{t(i)}$
there are no regions D and A in Figure \ref{distance_bad_words}, and region B is cut off in the middle.
Hence it is easy to see that we have
\begin{align*}
    S_{G_{k}^{t(i)}}g(\underline{x}^{(i-1)})-S_{G_{k}^{t(i)}}g(\underline{x}^{(i)})
    &< (14+(\beta^{t(i)}_k-\alpha^{t(i)}_k))\varepsilon.
\end{align*}

Then we have
\begin{align*}
    S_{[-N, N]\times [i,N]}g(\underline{x}^{(i)})-S_{[-N,N]\times[i,N]}g(\underline{x}^{(i-1)})
    &<\sum_{k=1}^{K_{t(i)}+1}(14+(\beta_k^{t(i)}-\alpha_k^{t(i)}))\varepsilon+\sum_{k=1}^{K_{t(i)}}|B_k^{t(i)}|(-1+2\varepsilon)\\
    &\leq 14(K_{t(i)}+1)\varepsilon+\sum_{k=1}^{K_{t(i)}}|B_k^{t(i)}|(-1+3\varepsilon)\\
    &\leq 14\varepsilon+(-1+17\varepsilon)|S^{t(i)}|,
\end{align*}
where the last inequality holds because $K_{t(i)}\leq |S^{t(i)}|$.

By the similar argument, the estimate on the remain region $\{(u_1, u_2)\mid u_1\in [-i, i], |u_1|\leq u_2\leq i-1\}=R^{t(i)}\cap \Lambda_{i-1}$ is bounded by
\begin{align*}
    S_{R^{t(i)}\cap \Lambda_{i-1}} g(\underline{x}^i)-S_{R^{t(i)}\cap \Lambda_{i-1}}g(\underline{x}^{i-1})
    &<\sum_{k=1}^{K_{t(i)}+1}(14+(\beta_k^{t(i)}-\alpha_k^{t(i)}))\varepsilon\\
    &\leq 14\varepsilon+15|S^{t(i)}|\varepsilon
\end{align*}
Then we have
\begin{align*}
    S_{R^{t(i)}}g(\underline{x}^{(i)})-S_{R^{t(i)}}g(\underline{x}^{i-1})<28\varepsilon+(-1+32\varepsilon)|S^{t(i)}|.
\end{align*}

For $R^{(b(t))}, R^{(r(i))}, R^{l(i)}$, by the similar argument we have
\begin{align*}
    S_{R^{\tau}}g(\underline{x}^{(i)})-S_{R^{\tau}}g(\underline{x}^{i-1})<28\varepsilon+(-1+32\varepsilon)|S^{t(i)}|.
\end{align*}
for $\tau \in \{b(i), r(i),l(i)\}$.
Combining all, we have
\begin{align*}
    S_{\Lambda_N}g(\underline{x}^{(i)})-S_{\Lambda_N}g(\underline{x^{(i-1)}})<112\varepsilon+(-1+32\varepsilon)\sum_{\tau}|S^{\tau}|
    =112\varepsilon+(-1+32\varepsilon) |S^{i}|.
\end{align*}

Dividing the both sides of \eqref{gap_with_periodic} by $(2N+1)^2$,, we have
\begin{align*}
    \frac{1}{(2N+1)^2}(S_{\Lambda_N}g(\underline{x})-S_{\Lambda_N}g(\tilde{\underline{x}}))
    &\leq\frac{1}{(2N+1)^2}\sum_{i=0}^{N}(112\varepsilon+(-1+32\varepsilon) |S^{i}|)+\frac{2(8N+1)}{(2N+1)^2}\\
    &\leq \frac{112\varepsilon (NL1)+2(8N+1)}{(2N+1)^2}-\frac{1}{2}\frac{1}{(2N+1)^2}\sum_{i=0}^{N}|S^i|\\
    &= \frac{112\varepsilon (N+1)+2(8N+1)}{(2N+1)^2}-\frac{1}{2}\frac{1}{(2N+1)^2}\#\{\boldsymbol{u}\in\Lambda_N\mid \sigma^{\boldsymbol{u}}\underline{x}\in I^c\}
\end{align*}

Hence we have
\begin{align*}
    \liminf_{N\to\infty}\frac{1}{(2N+1)^2}S_{\Lambda_N}g(\underline{x})+\frac{1}{2}\mu(I^c) &< \liminf_{N\to\infty} \frac{1}{(2N+1)^2}S_{\Lambda_N}g(\tilde{\underline{x}}).
\end{align*}
Since $\underline{x}$ is a generic point of $\mu$ and we see that there exists an invariant probability measure $\nu$ with support in $X$ by passing to a subsequence of the sequence of empirical measures for $\tilde{x}$, 
we have
\begin{align*}
    \int g d\mu<\int g d\nu,
\end{align*}
which complete the proof.
\end{proof}





\vspace*{33pt}

\noindent
\textbf{Acknowledgements.}~ 

The second author was partially supported by JSPS KAKENHI Grant Number 21K13816.

\vspace{11pt }
\noindent
\textbf{Use of AI tools.}~
The authors used ChatGPT (GPT-4, OpenAI) for English proofreading and improving the clarity of the manuscript. The AI tool was used solely to enhance language readability and did not influence the originality or the intellectual content of this research.

\vspace{11pt }
\noindent
\textbf{Data Availability.}~
Data sharing not applicable to this article as no datasets were generated or analyzed during the current study.


\bibliographystyle{alpha}
\bibliography{stability_hardsquare.bib}

\end{document}